\documentclass{amsart}
\usepackage[utf8]{inputenc} 
\usepackage{graphicx} 
\usepackage[leqno]{amsmath}
\usepackage{amsthm}
\usepackage{amssymb}
\usepackage[parfill]{parskip}
\usepackage{hyperref}
\newtheorem{theorem}{Theorem}[section]
\newtheorem{lemma}[theorem]{Lemma}
\newtheorem{corollary}{Corollary}[theorem]

\theoremstyle{definition}
\newtheorem{defn}[theorem]{Definition}

\title{Minors of asymptotically almost all sparse paving matroids}
\author{Will Critchlow}
\address{School of Mathematics and Statistics, Victoria University of Wellington,
New Zealand}
\email{william.critchlow@msor.vuw.ac.nz}

\begin{document}

\maketitle

\begin{abstract}
We use counting arguments to show that asymptotically almost all sparse paving
matroids contain an $H$-minor, where $H$ falls into one of several simple classes of
matroids. Furthermore the result holds for all $H$ in a larger class of matroids, if
we restrict to asymptotically almost all sparse paving matroids of fixed rank $r$
(where $r$ is necessarily no smaller than the rank of $H$).
\end{abstract}

\section{Introduction}
Letting $\mathbb{M}_n$ denote the set of matroids with groundset $[n]$, we say that
a property $\mathcal{P}$ holds for asymptotically almost all matroids if

$$\lim\limits_{n \rightarrow \infty} \left.\frac{\#\{M \in \mathbb{M}_n: M \text{
has property }\mathcal{P}\}}{\#\mathbb{M}_n} \right. = 1$$

In \cite{PvdP}, Pendavingh and van der Pol showed that asymptotically almost all
matroids contains a $H$-minor, for $H$ one of the matroids $U_{2,k},
U_{3,6},P_6,Q_6$ or $R_6$. 

Here we consider the related question of for which sparse paving matroids $H$ we can
say that asymptotically almost all sparse paving matroids contain an $H$-minor.
Interest in this problem is largely motivated by the conjecture of Mayhew, Newman,
Welsh and Whittle, in \cite{MNWW}, that asymptotically almost all matroids are
sparse paving. If true, the conjecture would imply that if asymptotically almost all
sparse paving matroids have an $H$-minor, then so must asymptotically almost all
matroids.

\subsection{Summary of results}

In this paper we show that a number of matroids are contained as minors in
asymptotically almost all sparse paving matroids. These include: \\

\begin{itemize}
\item Uniform matroids
\item Sparse paving matroids of rank $r$ whose non-bases all intersect in a single set
of $r-2$
elements.
\item Sparse paving matroids whose non-bases are pairwise disjoint
\item $W^3$
\end{itemize}

We note that the result could surely be extended to various other matroids with simple
structures of non-bases, but for reasons of space we shall only explicitly prove the
four above cases.

Furthermore, we show some matroids to be contained as minors in asymptotically
almost all sparse paving matroids of fixed rank $r$ (where obviously $r$ must be at
least the rank of the target minor). This is a weaker statement but we can prove it
for more target minors. Specifically, the statement holds for all sparse paving
matroids whose non-bases all contain at least one loose element, that is to say an
element not contained in any other non-bases.

In building the apparatus to prove the above results, we additionally show that
asymptotically almost all sparse paving matroids $M \in \mathbb{M}_n$ have at least
$(1+o(1))\frac{1}{4n}{n \choose r(M)}$ non-bases.

\section{Notation and preliminaries}

\subsection{Notation}

As is common, we write $\mathbb{M}_n$ to represent the class of matroids of size
$n$, and $\mathbb{M}_{n,r}$ the class of matroids of size $n$ and rank $r$. We shall
also use $m_n =
|\mathbb{M}_n|$ and $m_{n,r} = |\mathbb{M}_{n,r}|$

A matroid $M \in \mathbb{M}_{n,r}$ is \textit{sparse paving} if every $r$-set of
$[n]$ that is not a basis of $M$, is a circuit-hyperplane of $M$. We write
$\mathbb{S}_n$ for the class of sparse paving matroids of size $n$, and
$\mathbb{S}_{n,r}$ the class of sparse paving matroids of size $n$ and rank $r$.
$s_n = |\mathbb{S}_n|$ and $s_{n,r} = |\mathbb{S}_{n,r}|$.

We use the term \textit{non-basis} to describe a circuit of cardinality equal to the
rank of its matroid. Since this paper discusses non-bases exclusively in the context of
sparse paving matroids, the definition is here also equivalent to
\textit{circuit-hyperplane} (which we shall try to avoid using).

A property $\mathcal{P}$ is said to hold for \textit{asymptotically almost all} (or
a.a.a.) matroids if

$$\lim\limits_{n \rightarrow \infty} \left.\frac{\#\{M \in \mathbb{M}_n: M \text{
has property }\mathcal{P}\}}{m_n} \right. = 1$$

with similar definitions holding amongst sparse paving matroids, matroids of fixed
rank and sparse paving matroids of fixed rank. Equivalently, the property
$\mathcal{P}$ may be said to hold \textit{with high probability} (or w.h.p.).

We write $[n]$ for $\{1, 2, \ldots, n\}$ and $[n]^{(r)}$ means the $r$-subsets of
$[n]$.

\subsection{Preliminaries}

We note the following result which is fundamental to many of our proofs. \\

\begin{lemma}
For all $0 < \delta \le 1$, asymptotically almost all $n$-element sparse paving
matroids have rank $r$ in the range $((0.5-\delta)n,(0.5+\delta)n)$
\label{aaaranges}
\end{lemma}

A similar result already exists in the general matroidal case, and is included in
\cite{LOSW}.

\begin{proof}

Let $t_n$ be the number of $n$-element sparse paving matroids with rank outside
this range. This is bounded above by 

$$t_n < n\left(2^{{n \choose \lfloor(0.5-\delta)n\rfloor}}\right)$$

And we know from a result of Graham and Sloane \cite{GS} that

$$s_n > 2^{\frac{1}{n}{n \choose \lfloor n/2 \rfloor}}$$

The latter can easily be seen to outgrow the former - for example, by considering
binomial coefficients we get $$\frac{\log t_n}{\log s_n} = (1 + o(1)) \frac{{n
\choose \lfloor(0.5-\delta)n\rfloor}}{\frac{1}{n}{n \choose \lfloor n/2 \rfloor}} 
\le(1+o(1)) n\left(\frac{1-\delta}{1+\delta}\right)^{\delta n/2} $$

\end{proof}

\section{The Johnson Graph and Maximum extensions}

The main objective of this section is a proof of the result that asymptotically
almost all sparse paving matroids have at least $(1+o(1))\frac{1}{4n}{n \choose
r(M)}$ non-bases.

As observed in \cite{PvdP} the collections of non-bases of sparse paving matroids of
size $n$ and rank $r$ are in direct correspondence with the stable sets of the
Johnson graph $J(n,r)$. This was originally shown in a paper of Piff and Welsh
\cite{PiffW}. We shall find it useful to work in these terms for now.

\begin{defn} [Johnson graph]
The Johnson graph $J(n,r)$ has as its vertex set $[n]^{(r)}$. Vertices are joined if
and only if they intersect in exactly $r-1$ elements.
\end{defn}

A simple counting argument gives rise to the following bound on the size of a
stable set in $J(n,r)$.\\

\begin{lemma}
\label{maxstablelem}
A stable set in $J(n,r)$ has size at most $\frac{1}{n+1-r}{n \choose r}$.
\end{lemma}

Equality is achieved by the Steiner system $S(n,r,r-1)$, where it exists.

\begin{defn} [Maximal extension]
Given a stable set $I$ in a graph $G$, say $J$ is an \textit{extension} of $I$
if $J$ is
stable, and $I$ is contained in $J$. An extension $J$ of $I$ is called
\textit{maximal} if
there is no larger stable set containing $I$ (note that $I$ may have more than one
maximal extension under this definition). \\
\end{defn} 

\begin{theorem}
\label{extthm1}
\begin{itemize}
\item[(a)]Let $I$ be drawn uniformly at random from the stable sets of $J(n,r)$. For
any $\epsilon, \delta >0$ there is some $n_0$ large enough such that 
for any $n > n_0$ and $2 \le r \le n-2$, $I$ will have a
maximal extension of cardinality at least
$\frac{1}{(1+\epsilon)2n}{n \choose r}$, with probability at least $1-\delta$.
\item[(b)] Let $M$ be drawn uniformly at random from the sparse paving matroids with
groundset $[n]$. Asymptotically almost
always the non-bases of $M$, viewed as a stable set in the Johnson graph
J(n,r(M)), have a maximal extension of cardinality at least
$\frac{1}{(1+\epsilon)2n}{n \choose r}$.
\end{itemize}
\end{theorem}

The proof proceeds by a simple counting argument, for which we first require the
following lemma. This is based on a result of Byskov \cite{Bys} - for brevity we
present only a corollary of his result.\\

\begin{lemma}
The number of maximal stable sets of size $k$ in any $N$-vertex graph is at
most $$\lfloor N/k \rfloor^{k-\alpha}\lfloor N/k+1\rfloor^{\alpha}$$
where $\alpha = N \mod{k}$
\end{lemma}

(We use ``N'' for the number of vertices in the graph is to avoid confusion with
other uses of $n$: in the Johnson graph $J(n,r)$, the number of vertices is $N={n
\choose r}$.)

\begin{proof}[Proof of Theorem~\ref{extthm1}]
We show part (a). 
W.l.o.g. we can assume $\epsilon \le 1$.

We shall use the notation $K = \frac{1}{(1+\epsilon)2n}{n \choose r}$. 
Say $k'$ is the largest integer less than $K$.
Use $N = {n \choose r}$, and let $\mu_{n,r,k'}$ be the number of maximal stable
sets of size $k'$ on $J(n,r)$. By the above lemma we have  $$\mu_{n,r,k'} < \lfloor
N/k'+1\rfloor^{k'} < ((1+\epsilon)2n)^{K}  <
4n^{K}$$ with the last inequality using our
assumption that $\epsilon < 1$. Moreover, each of these maximal stable sets
contains $2^k$ subsets which are stable in $J(n,r)$. 
Now the value of the expression $$\lfloor N/k+1\rfloor^{k}$$ is an increasing
function for $k < N$, so we also have $$\mu_{n,r,k} < \lfloor N/k'+1\rfloor^{k'} <
4n^{K} $$ for any $k < k'$. \\

Now let $\mathcal{I}_{\epsilon}$ be the collection of stable sets of $J(n,r)$
whose maximal extension has cardinality less than $\frac{1}{(1+\epsilon)2n}{n
\choose r}$. Clearly each of these is a subset of one of the maximal stable
sets under consideration. We have         

\begin{multline}
|\mathcal{I}_{\epsilon}| < \sum\limits_{k = 1}^{\frac{1}{(1+\epsilon)2n}{n \choose
r}}4n^{K} 2^k = 4n^{\frac{1}{(1+\epsilon)2n}{n
\choose r}}.\sum\limits_{k =1}^{K} 2^k  =
4n^{K}.2^{\frac{1}{(1+\epsilon)2n}{n \choose
r}+1}  \\ 
\le 8n^{K+1}
\end{multline}

But for sufficiently large $n$, and assuming $2 \le r \le n-2$, we have
$K+1 <<
\frac{1}{2n}{n \choose r} - \log_2 8n$ - in fact, we can force
$K+1 < \frac{1}{2n}{n \choose r} - 2\log_2 8n$ ,
so$ |\mathcal{I}_{\epsilon}|< \frac{1}{8n}\left(2^{\frac{1}{2n}{n \choose
r}}\right)$. But we know that the number of
stable sets in $J(n,r)$ is at least $2^{\frac{1}{2n}{n \choose r}}$ \cite{GS}.
Now we
simply require $n >\frac{1}{8\delta}$ and we are done.
Part (b) of the result follows as a simple corollary of the above and Lemma
\ref{aaaranges}.
\end{proof}

\begin{corollary}
\label{extcor1}
\begin{itemize}
\item[(a)]Asymptotically almost all sparse paving matroids $M$ of rank $r$ have at
least $(1+o(1))\frac{1}{4n}{n \choose r}$ non-bases.
\item[(b)]Asymptotically almost all sparse paving matroids $M$ have at least
$(1+o(1))\frac{1}{4n}{n \choose r(M)}$ non-bases.
\end{itemize}
\end{corollary}

The corollary despite seeming obvious requires a little effort to obtain, and the
assistance of the following well-known theorem of Lubell, Yamamoto and Meshalkin.

\begin{defn}[Shadow]
For $\mathcal{A} \subseteq X^{(r)}$, the \textit{shadow} $\partial\mathcal{A}$ of
$\mathcal{A}$ is the set system $\{B \in X^{(r-1)}: B \cup \{i\} \in \mathcal{A},
\text{ for some } i \not\in B\}$. 
\end{defn}

\begin{theorem}[Local LYM]
For any set $X$ and $\mathcal{A} \subseteq X^{(r)}$, 

$$\frac{|\partial\mathcal{A}|}{{n \choose r-1}} \ge \frac{|\mathcal{A}|}{{n \choose
r}}$$
\end{theorem}

We can now tackle our corollary:

\begin{proof}[Proof of Corollary~\ref{extcor1} ] 

We start with part (a).

Let $\epsilon > 0$. For the sake of simplicity we continue to view $\mathcal{C}(M)$
as a stable set in the Johnson Graph.

As noted previously, a stable set may have more than one maximal extension. It now
becomes useful to identify each stable set with precisely one maximal extension,
such that these extensions define a partition on the space stable sets. To achieve
this we define some total ordering relation $<'$ on all stable sets of $J(n,r)$ with
the property that $|J| > |I| \implies J >' I$.\footnote{e.g. let $>_k'$ be a
lexicographical ordering on sets of size $k$, and say $J >' I$ if $|J| > |I|$ or if
$|J| = |I| = k$ and $J >_k' I$.}

 Let $m'(I)$ denote the extension of a stable
set $I \subset J(n,r)$ which is maximal under $<'$. 

We note the following useful property of this definition:

\begin{equation}
\label{k7zK}
I \subseteq I' \subseteq m'(I) \implies m'(I')=m'(I)
\end{equation}

Suppose $M$ is chosen uniformly at random from the $n$-element sparse paving
matroids of
rank $r$, and let $I=\mathcal{C}(M)$. Effectively then we have drawn $I$ uniformly
at random from the stable sets of $J(n,r)$. Letting $Y$ be the collection of all
maximal
stable sets in the Johnson graph $J(n,r)$, and recalling that the unique maximal
extensions partition the space of stable sets, we can say:

$$\mathrm{Pr}\left(|I|<(1-\epsilon)\frac{1}{4n}{n \choose r}\right) = \sum\limits_{J
\in Y}\mathrm{Pr}\left(m'(I) =J\right)\mathrm{Pr}\left(|I|<(1-\epsilon)\frac{1}{4n}{n
\choose r}: m'(I)=J\right)$$

We claim that  $\forall J \in Y$, $\mathrm{Pr}(|I|<(1-\epsilon)\frac{1}{4n}{n
\choose r}) : m'(I)=J)
\rightarrow 0$ as $n$ grows. We do this by working in a different probability space:
let $\mathrm{Pr}_J(X(I))$ denote the probability that a statement $X(I)$ is true for
$I$ drawn uniformly at random from the subsets of a stable set $J$. Of course
$\mathrm{Pr}(|I|<(1-\epsilon)\frac{1}{4n}{n \choose r}) : m'(I)=J) = 
\mathrm{Pr}_J(|I|<(1-\epsilon)\frac{1}{4n}{n \choose r}) : m'(I)=J)$, since $m'(I) =
J \implies I \subset J$

Now Bayes's Theorem gives us

\begin{multline}
\label{6YW3}
\mathrm{Pr}\left(|I|<(1-\epsilon)\frac{1}{4n}{n \choose r}) : m'(I)=J\right) \\
=\mathrm{Pr}_J\left(|I|<(1-\epsilon)\frac{1}{4n}{n \choose r}) : m'(I)=J\right) \\
= \frac{\mathrm{Pr}_J(|I|<(1-\epsilon)\frac{1}{4n}{n \choose
r})\mathrm{Pr}_J(m'(I)=J:|I|<(1-\epsilon)\frac{1}{4n}{n \choose
r})}{\mathrm{Pr}_J(m'(I)=J)}
\end{multline}

For a maximal stable set $J$, define $\mathcal{A}_r(J)$ to be the collection of
$r$-subsets of $J$ for which $J$ is \textbf{not} a maximal extension:
$\mathcal{A}_r(J) = \{I \in J^{(r)}: m'(I) \neq J\}$.

\ref{k7zK} tells us that $A_{r-1}(J) \supseteq \partial A_r(J)$.

Applying Local LYM, we see $r < s \implies \frac{\mathcal{A}_r(J)}{{|J| \choose r}} \ge
\frac{\mathcal{A}_s(J)}{{|J| \choose s}}$, which in turn implies that
$\mathrm{Pr}_J(m'(I)=J:|I|=r) \le \mathrm{Pr}_J(m'(I)=J:|I|=s)$. It follows that:

$$\frac{\mathrm{Pr}_J(m'(I)=J:|I|<(1-\epsilon)\frac{1}{4n}{n \choose
r})}{\mathrm{Pr}_J(m'(I)=J)} \le 1$$

Applying this information to \ref{6YW3} tells us that
$$\mathrm{Pr}\left(|I|<(1-\epsilon)\frac{1}{4n}{n \choose r}) : m'(I)=J\right) \le
\mathrm{Pr}_J\left(|I|<(1-\epsilon)\frac{1}{4n}{n \choose r}\right)$$

And it's easily seen by consideration of binomial coefficients that 

$$\mathrm{Pr}_J\left(|I|<(1-\epsilon)\frac{1}{4n}{n \choose r}\right) \rightarrow 0$$

So for any $\delta < 0$, we have that for large enough $n$

$$\mathrm{Pr}\left(|I|<(1-\epsilon)\frac{1}{4n}{n \choose r}) : m'(I)=J\right) <
\delta$$

and so 

$$\mathrm{Pr}\left(|I|<(1-\epsilon)\frac{1}{4n}{n \choose r}\right) <
\delta\sum\limits_{J \in Y}\mathrm{Pr}\left(m'(I)=J\right) = \delta$$

This has shown part (a). Part (b) follows in a similar way to Theorem~\ref{extthm1}:
all bounds in the above proof are tighter the closer $r$ gets to $\frac{n}{2}$, so
the results hold for $M$ drawn from all sparse paving matroids with rank in the
interval $[r,n-r]$, and a.a.a the rank is in that interval. 

\end{proof}

\section{A counting approach to the minor inclusion question}

Theorem~\ref{bigthm} will provide a condition under which a matroid $H$ will
be contained as a minor in a.a.a.\ sparse paving matroids. Essentially we will show
that if the set of non-bases of a matroid contains many subsets isomorphic to the
non-bases of $H$, then almost certainly at least one such instance forms an
$H$-minor (i.e. by having no additional non-bases). \\

A consequence of this is that we can avoid further worrying about inclusion of the
bases of $H$ (or as they might better be
considered in this situation, non-non-bases!)

Before describing this theorem we must introduce some definitions.

\begin{defn}[Line structure]
A \textit{line structure} $L$ of rank $r$ on $n$ elements is defined to be a
collection of sets of cardinality $r$ obeying the rule that $l_1,l_2 \in L \implies
|l_1 \cap l_2| \le r-2$. The collection of non-bases of a sparse paving matroid $M$
forms a line structure, which we may denote by $L(M)$.
Equally every line structure $L$ defines a matroid, $M(L)$,
whose non-bases are isomorphic to $L$, with groundset exactly the union of those
non-bases.

We say a line structure $L_1$ of rank $r$ contains another $L_2$ of rank $r$ (and
write $L_2 \subseteq L_1$) if a subset of $L_1$ is isomorphic to $L_2$. 

\end{defn}

\begin{defn}[Element-disjoint]
Line structures $L_1$ and $L_2$ are \textit{element-disjoint} if the unions of their
elements are disjoint:

$$\left(\bigcup\limits{l \in L_1}\right) \cap \left(\bigcup\limits{l \in L_2}\right)
= \emptyset$$

\end{defn}

For the following theorems we use ``contain'' in the sense of subhypergraphs; in
other words, by ``$M/A$ contains a copy of $L$'' we mean that the non-bases of
$M/A$, viewed as a hypergraph, contains a subhypergraph isomorphic to $L$. We note
that this is \textbf{not} equivalent to $M/A$ containing $L$ as a minor, since $M/A$
may contain more hyperedges on the same elements.

\begin{theorem}
\label{bigthm}
Let $H$ be a fixed sparse paving matroid. Consider drawing $M$ uniformly at random
from $\mathbb{S}_n$. We say that a line structure $L$ is \textit{abundant} in sparse
paving
matroids if, for any $m \in \mathbb{N}$, there is asymptotically almost always a set
$A \subset [n], |A| < r(M)-r_H$ such that $M/A$ contains $m$ element-disjoint
copies of $L(H)$.

If $L(H)$ is abundant in sparse paving matroids, then asymptotically almost all
sparse paving matroids contain $H$ as a minor.
\end{theorem}

It is easier to prove directly the following theorem, of which Theorem \ref{bigthm}
is an immediate consequence.\\

\begin{theorem}
\label{bigthm2}
Let $H$ be a fixed sparse paving matroid. Let $M$ be a matroid drawn randomly from
$\mathbb{S}_{n,m}$, where $\mathbb{S}_{n,m}$ denotes those matroids in $\mathbb{S}_n$
for which there exists $A \subset [n], |A| <
r(M)-r_H$ such that $M/A$ contains $m$ element-disjoint copies of $L(H)$.

For any $\epsilon > 0$, there exists $m(\epsilon)$ such that if $m > m(\epsilon)$,
then with probability at least $1-\epsilon$ $M$ will contain $H$ as a minor.
\end{theorem}

\begin{proof}

Assume w.l.o.g. that $\epsilon < 1$.

We aim to assign $M$ a parent matroid $T \in
\mathbb{S}_{n,m}$, so that every choice of $M$ has precisely one such parent. We
show then that for any parent, at most proportion $\epsilon$ of its children are
$H$-free. The same will then be true for the union of offspring of all parents,
which is of course equal to the entirety of $\mathbb{S}_{n,r,m,H}$.

Let $n_H$ denote the number of elements of $H$ and $n_L(H)$ the number of elements
of $H$ that are included in at least one non-basis of $H$ (i.e. the size of the
union of the
non-bases of $H$). Let $r_d = r(M)-r_H$.

We treat our groundset as $[n]$. Imagine taking a collection $U = \{(U_1,L_1),
\ldots, (U_{m-1},L_{m-1})\}$ where the $L_i$ denote $m$ pairwise element-disjoint
copies of the line structure $L(H)$ on $[n]^{(r_H)}$ and the $U_i$ a collection of
$m$ pairwise disjoint $n_H$-subsets of $[n]$ such that, for each pair $(U_i,L_i)$,
the lines of $L_i$ are contained in $U_i$. We define $\mathcal{U}(n,m,L(H))$ to be
the set of all collections of this type.

If $M$ contains $m$ pairwise element-disjoint copies of $L(H)$ then we may assume it
also contains a collection $U \in \mathcal{U}(n,m,L(H))$ (we simply require that
$n(M)>mn_H$ in order to be able to construct such a collection).

Fix some total ordering $*(n,m,L(H))$ on the set $\mathcal{U}(n,m, L(H))$. For ease we
say that $U >^* W$ if $U$ is above $W$ in this ordering. 

Also fix a total ordering $+(n,r_d)$ on $[n]^{(r_d)}$. Write $A >^+ B$  if $A$ is
above $B$ in this ordering.

The reason for all this careful preparation is to ensure each matroid has a single
well defined parent, which means we avoid duplicate counting in the later steps.
Note we can be as careless as we like over the actual orderings: any
total ordering will do (although it is probably most intuitive to imagine that the
orderings are of a lexicographical nature).

We now build the sets from which our "parents" will be taken. For any $A \in
[n]^{(r_d)}$, let $\mathcal{T}_A \subseteq \mathbb{S}_{n,m}$ be those sparse paving
matroids for which:

\begin{itemize}
\item[(i)] $T/A$ contains a collection of $m$ pairwise element-disjoint copies of
$L(H)$ (and hence also one of the paired collections of $\mathcal{U}(n,m,L(H))$).
\item[(ii)] $A$ is the maximal member of $[n]^{(r_d)}$ for which this is true: that
is, $B >^+ A \implies T/B$ does not contain $m$ element disjoint copies of $L(H)$.
\end{itemize}

Now for any $U = \{(U_1,L_1), \ldots, (U_{m-1},L_{m-1})\} \in
\mathcal{U}(n,m,L(H))$, let us define $\mathcal{T}_{A,U}$ to be those members of
$\mathcal{T}_A$ for which:

\begin{itemize}
\item[(i)] $U$ is the maximal member of $\mathcal{U}(n,m,L(H))$ (under our ordering
$*$) that appears in $T/A$.
\item[(ii)] On each $U_i$, $T/A$ contains no circuits other than those in $L_i$:
hence the restriction $(T/A)|U_i \cong H$.
\end{itemize}

We know there exists some $ A \in [n]^{(r_d)}$ such that $L(M/A)$ contains $m$ disjoint
copies of $L(H)$. We can assume w.l.o.g. that $A$ be the maximal such set under our
ordering relation $>^+$. Now assuming that $n(M) - r_d>mn_H$, we also have some $V
\in \mathcal{U}_{n,m,L(H)}$ contained in $M/A$. Let $V$ be the maximal such
collection and denote its members (the $n_H$-subsets) as $V_0, \ldots, V_{m-1}$,
ordered lexicographically.

With some care we define 
$$\Sigma_i = \{A \cup S: S \in V_i^{(r_H)}\}$$

The $\Sigma_i$ represent collections of possible non-bases and provide a simple
condition for $M$ to contain an $H$-minor - we
simply require one of the $\Sigma_i$ to contain no non-bases of $M$ other than those
of form $A \cup l: l \in L_i$. 

We now can set the parent $T(M)$ of $M$ to be $T \in \mathcal{T}_{A,V}$ formed by
deleting from all $\Sigma_i$ all circuits of $M$ that are not part of the copies of
$L(H)$ of $V$ (such that the restriction of $T/A$ to any $V_i$ will be isomorphic to
$H$).

Although the proof we follow is strictly speaking a counting argument, it perhaps is
more naturally understood in the language of probability, which is how we shall
proceed. Let $X_i$ be the event that the restriction of $M/A$ to $S_i$
forms a $H$-minor and $Y_i$ the converse; let $Z_k = Y_0 \cup \ldots \cup Y_k$.

We claim the following: 

\begin{equation}
\label{dPe4}
\mathrm{Pr}(Z_k) \le \left(\frac{s_{n_H,r_H}-1}{s_{n_H,r_H}}\right)^{k+1}
\end{equation}

This is sufficient to prove our theorem: for any $\epsilon > 0$ we could now achieve
$\mathrm{Pr}(M \not\succeq H) <
\epsilon$ simply by setting $m(\epsilon) \ge -\log \epsilon/(\log(s_{n_H}-1) - \log
(s_{n_H}))$.

We shall in fact show that for all $0 \le k \le m-1$, $\mathrm{Pr}(X_k|Z_{k-1}) \ge
\frac{1}{s_{n_H,r_H}}$. Clearly this is enough to imply \ref{dPe4} by simple
inductive reasoning.

Now further define 

Take any $A \in [n]^{(r_d)}, U = \{(U_1,L_1), \ldots, (U_{m-1},L_{m-1})\} \in
\mathcal{U}_{n,m,L(H)}$ and $T \in \mathcal{T}_{A,U}$. We assume the $U_i$ take a
lexicographical ordering (or any total ordering fixed universally on $[n]^{(n_H)}$).
The children of $T$ are contained in the collection of sparse paving matroids which
can be formed by adding to $T$ circuits of the form $A \cup S$, such that $S$ is
contained in some $U_i$.\\
Consider adding such a collection of circuits $\mathcal{C}_k^-$ where we add the
temporary restriction that $S$ can only be chosen from one of $U_1, \ldots, U_{k-1},
U_{k+1}, \ldots, U_{m-1}$. Not every such addition of circuits creates a child: they
may conflict with other circuits of $T$, or they might compromise the maximality of
$U$ or $A$. Suppose however that having chosen this $\mathcal{C}_k^-$, there is also
some further choice $\mathcal{C}_k$ of circuits we can add, this time restricting
$S$ to $U_{k}$, such that we are left with a child of $T$. Then we are certainly
also left with a child of $T$ if we add no circuits where $S$ is chosen from $U_k$.
And contracting this child by $A$ and restricting to $U_i$ gives an $H$-minor. So at
this stage at least one of a maximum $s_{n_H,r_H}$ choices of $\mathcal{C}_k$ leads
to an $H$-minor. Since this holds for every choice of $\mathcal{C}_k^-$, we can
actually say that

$$\mathrm{Pr}(X_k|T(M)=T) \ge \frac{1}{s_{n_H,r_H}}$$

But in fact the same logic prevails even were we to at the previous stages demand
that we add at least one circuit where $S$ is drawn from $U_i$, for all $0 \le i \le
k-1$ - essentially restricting ourselves to considering matroids for which $Z_{k-1}$
holds. So we can in fact say that

$$\mathrm{Pr}(X_k|Z_{k-1},T(M)=T) \ge \frac{1}{s_{n_H,r_H}}$$

And now, writing $\mathcal{U}$ for $\mathcal{U}(n,m,L(H))$, we have

$$\mathrm{Pr}(X_k|Z_{k-1}) = \sum\limits_{U \in \mathcal{U}}\sum\limits_{T
\in\mathcal{T}_U}\mathrm{Pr}(X_k|Z_{k-1},T(M)=T)\mathrm{Pr}(T_(M)=T|Z_{k-1})  $$

$$ \ge \sum\limits_{U \in \mathcal{U}}\sum\limits_{T
\in\mathcal{T}_U}\frac{1}{s_{n_H,r_H}}\mathrm{Pr}(T(M)=T|Z_{k-1}) $$
$$ \ge \frac{1}{s_{n_H,r_H}}\sum\limits_{U \in \mathcal{U}}\sum\limits_{T \in
\mathcal{T}_U}\mathrm{Pr}(T(M)=T|Z_{k-1}) =\frac{1}{s_{n_H,r_H}}$$

We may cancel the sum of probabilities because each matroid has a well-defined
parent, meaning it is counted precisely once in the above summation.

\end{proof}

Note that the above method can be easily adapted to provide an equivalent result
restricted to matroids of any given rank $r \ge r_H$, by requiring that we contract
by a set $A$ of given size. In particular, if we do not contract at all, we get the
following result:\\

\begin{lemma}
\label{bigthmfixrank}
Let $H$ be a fixed sparse paving matroid. Let $M$ be a matroid drawn randomly from
those matroids in $\mathbb{M}_{n,r_H}$ which contain $m$ element-disjoint copies of
$L$.

For any $\epsilon > 0$, there exists $m(\epsilon)$ such that if $m > m(\epsilon)$,
then with probability at least $1-\epsilon$ $M$ will contain $H$ as a minor.
\end{lemma}

We state this because it proves useful to be able to quote the theorem in this form
when pursuing results for matroids of fixed rank in  \ref{ramseychapter}.

\subsection{Abundant minors}
Recall our earlier definition of \textit{abundance}:

\begin{defn} [Abundance]
A line structure $L$ is \textit{abundant} in sparse paving matroids if for any $m
\in \mathbb{N}$, for a.a.a.\ $n$-element sparse paving matroids $M$ there exists a
set $A \in [n]^{(r(M)-r_H)}$ such that $M/A$ contains $m$ element-disjoint copies
of $L$.
\end{defn}

\begin{defn} [Extremal density]
The \textit{extremal density} of a line structure $L$ of rank $r$ in
$\mathbb{S}_{n,r}$ is 

$$\mathrm{ex}(n,L) = \mathrm{max}\{p: \exists M, |\mathcal{C}(M)| \le
p\left(\frac{1}{n}{n \choose r}\right), \mathcal{C}(M) \not\supseteq L\}$$

taken over matroids $M \in \mathbb{S}_{n,r}$.
\end{defn}

In simpler terms, it is the maximal density of non-bases a sparse paving matroid can
have before its non-bases must contain a copy of $L$.

We wish to show that certain conditions of extremal density can imply abundance.

\begin{theorem}
\label{arbitrarilymany}
Let $L$ be a line structure of rank $r_L$ on $n_L$ elements, and $k$ any integer.
Let $L^k$ denote the line structure consisting of $k$ element-disjoint copies of
$L$. Then
$$\lim\limits_{n \rightarrow \infty}\mathrm{ex}(n,L^k) = \lim\limits_{n \rightarrow
\infty}\mathrm{ex}(n,L)$$
\end{theorem}

\begin{proof}
Our proof is inductive. Suppose that 

$$\lim\limits_{n \rightarrow \infty}\mathrm{ex}(n,L^{k-1}) = \lim\limits_{n
\rightarrow \infty}\mathrm{ex}(n,L) = p$$

Clearly $\mathrm{ex}(n,L^{k}) \ge \mathrm{ex}(n,L)$, so we need to show that for any
$\epsilon > 0$, there exists $n(\epsilon)$ such that $n > n(\epsilon) \implies
\mathrm{ex}(n,L^k) < p + \epsilon$.
Let $M$ be an $n$-element sparse-paving matroid of rank $r_L$ with density of
non-bases (the fraction of all $r_L$-sets of $M$ which are non-bases) being $\frac{p
+ \epsilon}{n}$, and suppose that $n$ is large enough that
$\mathrm{ex}(n,L^{k-1}), \mathrm{ex}(n,L) < p + \epsilon/2$. So $M$ must contain
some copy of $L$. But the number of non-bases intersecting with a line in our copy
of $L$ is at most $n_L\frac{1}{n-r_L}{n \choose r_L-1}$
(by~\ref{maxstablelem}). For large enough $n$, this is less than
$\frac{\epsilon}{2n}{n \choose r_L}$, and so even relaxing all these non-bases, we
are left with a density of non-bases greater than $\mathrm{ex}(n,L^{k-1}$). Hence we
must also contain some copy of $L^{k-1}$ disjoint to our copy of $L$, and the union
of these is a copy of $L^k$.
\end{proof}

\begin{lemma}[Abundance Lemma]
\label{abundancelem}
Let $H$ be a sparse paving matroid.
Suppose the following holds: $\exists \epsilon > 0,n_0 \in \mathbb{N},$ such that
for any $n > n_0$ we have $\mathrm{ex}((n,L(H))) < \frac{1}{8}-\epsilon$. Then $H$
is abundant in the sparse paving matroids.
\end{lemma}

\begin{proof}
Firstly, by Lemma~\ref{arbitrarilymany} we can assume w.l.o.g. that for $n > n_0$ we
have $\mathrm{ex}((n,L(H)^m)) < \frac{1}{8}-\epsilon$.

We know from Corollary~\ref{extcor1} that asymptotically almost all $M
\in \mathbb{S}_{n,r}$ have 
$$|\mathcal{C}(M)| > \left(\frac{1}{4}-\frac{\epsilon}{2}\right)\left(\frac{1}{n}{n
\choose r}\right)$$

and that also for any $\delta > 0$ asymptotically almost always
$\left(\frac{1}{2}-\delta\right)n \le r(M) \le \left(\frac{1}{2}+\delta\right)n$
(Lemma~\ref{aaaranges})\\
Consider only sparse paving matroids $M$ that meet the above two conditions, and
assume w.l.o.g. $n \ge n_0$. 
Contracting $M$ by $r-r_H$ randomly chosen elements, the expected density of
non-bases in our new
matroid is equal to the density of non-bases in $M$, so there must be some set $A$ of
$r-r_H$ elements such that 

$$|\mathcal{C}(M/A)| >
\left(\frac{1}{4}-\frac{\epsilon}{4}\right)\left(\frac{1}{n}{n' \choose
r_H}\right)$$

where $n' = n +r_H-r$.

Note that $n' > (1-\delta)\frac{n}{2} \implies \frac{1}{n} >
\frac{(1-\delta)}{2}\frac{1}{n'}$. This means by setting $\delta$ small enough we
can force \\

$$|\mathcal{C}(M/A)| >
\left(\frac{1}{8}-\frac{\epsilon}{2}\right)\left(\frac{1}{n_0}{n' \choose
r_H}\right)$$

But now we know $M/A$ has to contain a copy of
$L(H)^m$, and so $m$ copies of $L(H)$.

\end{proof}

We can now identify various matroids which have line structures abundant in the
sparse paving matroids, and hence are contained as a minor in asymptotically almost
all sparse paving matroids.

Firstly, using the fact the $L(U_{t,k})$ is empty and so trivially abundant in the
sparse paving matroids:\\

\begin{theorem}
For any integers $t,k,$, the uniform matroid $ U_{t,k}$ is contained as a minor in
asymptotically almost all sparse
paving matroids.
\end{theorem}

\begin{theorem}
Let $H$ be sparse paving matroid of rank $r$ in which all non-bases are pairwise
disjoint. $L(H)$ is abundant in the sparse paving matroids.
\end{theorem}

\begin{proof}
Clearly any matroid with non-zero density of non-bases contains a single non-basis.
Apply Lemma~\ref{arbitrarilymany}.
\end{proof}

\begin{corollary}
Let $M$ be any sparse paving matroid  of rank $r$ in which all non-bases are
pairwise disjoint. Asymptotically almost all sparse paving matroids contain $M$ as a
minor.
\end{corollary}

\begin{theorem}
Let $H$ be a sparse paving matroid of rank $r$ in which all non-bases meet in a set of
size $r-2$. $L(H)$ is abundant in the sparse paving matroids.
\end{theorem}

\begin{proof}
This is also relatively easy.

Suppose $H$ is a sparse paving matroid on $n_H$ elements with rank $r_H$ and all its
lines intersect in one $(r_H-2)$-subset of $[n_H]$. Let $M$ be an $n$-element sparse
paving matroid of rank $r_H$ with density of non-bases being $\frac{\epsilon}{n}$,
for some $\epsilon > 0$. Now given any $(r-2)$-subset of $[n]$, the expected number
of non-bases of $M$ which contain that set is $\frac{\epsilon}{n}{(n+2-r_H) \choose
2}$, so there must be some $(r_H-2)$-subset $A$ such that 

$$|\{C \in \mathcal{C}(M): C \supseteq A\}| \ge \frac{\epsilon}{n}{n+2-r_H \choose
2} \ge \frac{\epsilon}{n}{n/4 \choose 2}$$ 

For large enough $n$ this number will be greater than $|L(H)|$, and hence
$\mathrm{ex}((n,L(H))) < \epsilon$.
\end{proof}

\begin{corollary}
Let $M$ be any sparse paving matroid  of rank $r$ in which all non-bases meet in a
set of size $r-2$. Asymptotically almost all sparse paving matroids contain $M$ as a
minor.
\end{corollary}

We next consider the rank 3 whirl, $W_3$.\\

\begin{theorem}
There  exists a matroid $H$, containing $W_3$ as a minor, such that $L(H)$ is
abundant in the sparse paving matroids.
\end{theorem}

\begin{proof}
We find a matroid $H \succeq W_3$ such that for $\epsilon =
\frac{1}{2}\left(\frac{1}{8} - \frac{1}{9}\right)$, and large enough $n$,
$\mathrm{ex}((n,L(H))) < \frac{1}{8}-\epsilon$. Consider a matroid $M$  in
$\mathbb{S}(n,10)$
with at least density $(\frac{1}{4}-\epsilon)$ of non-bases. Let $Y$ be the random
variable denoting the number of non-bases which contain a set of 8
elements $\{a_1',a_2',\ldots,a_8'\}$, when randomly choosing that set from the
groundset elements.

Now consider drawing at random a circuit $C= \{a_1, \ldots, a_{10}\} \in
\mathcal{C}(M)$, and randomly discarding one element (say $a_{10}$) to form the set
$C^- = \{a_1,\ldots,a_9\}$. For each $i \in \{1, \ldots, 9\}$, let $X_i$ be the
derived variable given by the number of non-bases of $M$ containing $C^- - a_i$. Now
for any choice of $i$ we have:

$$ \mathrm{E}(X_i) = \sqrt{\mathrm{E}(Y^2)} \ge \mathrm{E}(Y)  \ge \left(\frac{1}{8}
-\epsilon\right) \left(\frac{n-9}{2}\right) $$ 

with the first inequality an application Jensen's inequality. But then 

$$\sum\limits_{i=1}^9 \mathrm{E}(X_i) >
\left(\frac{n-9}{2}\right)\left(1+\frac{1}{72}\right) > \frac{n}{2}$$

(the final inequality assuming sufficiently large $n$).

So there is some choice of $C = \{a_1, \ldots, a_{10}\}$ such that  $\sum_{\alpha
\in C^-} X_{\alpha} > n/2$. But by
the pigeonhole principle we must have some element $b \not\in C$ such that $M$
contains more than one circuit of form 

$$C-\{a_10, a_i\} \cup \{b,y\}, i \in \{1, \ldots, 9\}, y \not\in C$$

Note that the choices of deleted $a_i$ must be different, else the non-bases would
differ in only one element. Say without loss of generality we have non-bases $C' =
\{a_1,\ldots,a_8,b,c\}$ and $C''= \{a_1,\ldots,a_7,a_9,b,d\}$. So the line structure
$L$ isomorphic to $\{C,C',C''\}$ has extremal density less than
$\frac{1}{8}-\epsilon$ and is abundant in the sparse paving matroids.
And letting $H = M(L)$, we see that $H/\{a_1, \ldots, a_7\} \cong W_3$.

\end{proof}

\begin{corollary}
Asymptotically almost all sparse paving matroids contain $W_3$ as a minor.
\end{corollary}

We could continue in this vein, addressing target minors one by one, and doubtless
there are many which would fall to arguments like the above (albeit perhaps
increasingly complicated arguments!) For now we shall leave that area open.

\section{A Ramsey-theoretical approach} \label{ramseychapter}

We have already seen that sparse paving matroids can be identified with hypergraphs
whose
edges are the non-bases of the matroid. Moreover the independent (or
dependent) sets of any given rank can be viewed as a hypergraph. This opens up
varied possibilities for using Ramsey-theoretical techniques as a means to answering
minor-inclusion questions. 

The relevant result of Ramsey theory is the general Ramsey Theorem for hypergraphs.

\begin{theorem}[Ramsey]
Let $G$ be a $r$-uniform hypergraph on $n$ vertices with edges $k$-coloured - that
is, taking colours in some sets $\{c_1, c_2, \ldots, c_k\}$. Then for any positive
integers $t_1,t_2,\ldots, t_k$ there exists some number $n(t_1, t_2, \ldots, t_k)$
such that for some $i \in \{1, 2, \ldots, k\}, G$ contains a $c_i$-coloured
$K_{t_i}^{(r)}$.
\end{theorem}

\subsection{Uniform matroids}
In the previous section we proved at length that various matroids, including all
uniform matroids, are
contained in a.a.a.\ sparse paving matroids.However the result for uniform matroids
can also be achieved
almost immediately using Ramsey theory.

\begin{theorem}
\label{uniframseysp}
Asymptotically almost all sparse paving matroids contain $U_{t,k}$ as a minor.
\end{theorem}

\begin{proof}
Consider $\mathbb{S}_{n,r}$, the collection of sparse paving matroids of rank $r$ on
$n$
elements. Let $M \in \mathbb{S}_{n,r}$ and consider $M/A$, where $A = \{1, 2, \ldots,
(r-t)\}$. We note that with high probability $r < 3n/4$ (due to Lemma~\ref{aaaranges})
and we may consider only these cases. Hence we can assume that $M/A$ has at least
$n/4$ elements. Now we are done if, on some set of $k$ elements,
$M/A$ contains no non-bases. But now the set $\mathcal{C}(M)$ of
non-bases of $M/A$ is a hypergraph, and hence by Ramsey if $n/4$ is
sufficiently large (specifically, $n > 4R_r(r+1,k)$) then $\mathcal{C}(M)$ must
contain either a clique $K_{r+1}^{(r)}$ or stable set $E_k^{(r)}$. The former case
is impossible, as the elements of the clique would form a hyperplane of cardinality
$r+1$ in $M/A$, contradicting sparse pavingness. So there is some set $B$ of $k$
elements on which $\mathcal{C}(M)$ is empty. Then $(M/A)|B \cong U_{t,k}$.
\end{proof}

We'll now introduce an area which enables us to prove a minor-inclusion result for a
substantial class of matroids, albeit in the limited arena of fixed-rank sparse
paving matroids. The value of the result is perhaps less than the interest of the
Ramsey-theoretical methods used to obtain it.

\subsection{Loose elements and tied non-bases}

\begin{defn}[Loose element, Tied non-basis]
Let $M$ be a sparse paving matroid. An element $x$ in a non-basis $C$ of
$M$ is called \textit{loose} if no other non-basis of $M$ contains $x$. Moreover, a
non-basis $C$ is called \textit{tied} if it contains no loose elements.
\end{defn}

The main theorem of this chapter is the following.\\

\begin{theorem}
\label{aaaloose}
Let $H$ be a sparse paving matroid with no tied non-bases. Then asymptotically almost
all matroids of rank $r_H$ contain $H$ as a minor.
\end{theorem}

The proof will use Ramsey theory to establish that all sparse paving matroids of
sufficient size and rank $r$ must fall into at least one of two categories, within
each of which
asymptotically almost all matroids contain $H$ as a minor. One category is those
matroids which contains (as a subgraph) sufficiently many copies of $L(H)$, which we
show by demonstrating the inclusion of a structure which we shall call a
\textit{fort}. The other is matroids some contraction of which contains
sufficiently many copies of a structure which we shall call a \textit{moat}. We shall
first define those structures.

\begin{defn}[Fort]
Let $M$ be a sparse paving matroid of rank $r$. A collection of elements $X$ forms a
\textit{fort} of $M$ if for every $A \in X^{(r-1)}, \exists C \in \mathcal{C}(M): C
= A \cup e, e \not\in X$.
\end{defn}

\begin{defn}[Moat]
Let $M$ be a sparse paving matroid of rank $r$. A collection of elements $X$ forms a
\textit{moat} of $M$ if $M$ contains no non-bases whose elements intersect
with $X$ in $r-1$ elements, in other words if $C \in \mathcal{C}(M), |C| = r
\implies |X \cap C| \neq r-1$.  
\end{defn}

We note that forts and moats are in some sense inverse to one another - in one case,
we have the maximum possible number of  circuit hyperplanes of form $A \cup e, e
\not\in X$; in the other, the minimum possible number. In either case, the structure
provides a controlled space in which to demonstrate the existence (or in the case of
moats, probable existence) of minors. 

We'll first address the case of moats. The usefulness of moats is that if we can fix
some moat and everything outside the
moat, then we can allow any legitimate structure of non-bases within the
moat (i.e\.
those contained entirely within the moat $X$) and the whole structure will remain
sparse paving. However in this instance we will want to make further restrictions on
the interior of the moat.

\begin{defn}[Empty moat]
Let $M$ and $H$ be sparse paving matroids of rank $r$. We say a moat $X$ of $M$ is
\textit{empty} if there are no non-bases in the interior of the moat, that is to say
$X$ contains no non-bases of $M$. We also shall say $X$ is \textit{$H$-good} if the
non-bases of $M$ contained in $X$ are isomorphic to a subhypergraph of the line
structure of $H$. 
\end{defn}
We note that an empty moat $X$ is $H$-good for any choice of $H$.

\begin{lemma}
\label{hgoodlemma}
Let $H$ be a sparse paving matroid with $n_H$ elements and rank $r$. 
We let $\mathbb{S}_{n,r,m,H}$ denote the members of $\mathbb{S}_{n,r}$ which contain
at least $m$ pairwise disjoint $H$-good moats of size $n_H$.
Let $\epsilon > 0$. Then there exists an integer $m(\epsilon)$ such that for large
enough $n$ the following holds: let $M$ be drawn randomly from
$\mathbb{S}_{n,r,m(\epsilon),H}$. With probability at least $1-\epsilon$, $M$
contains $H$ as a minor.
\end{lemma}

The proof of this is very similar to Theorem~\ref{bigthm2}. The main difference is
that we are now demanding that our moats contained a subhypergraph of $L(H)$,
whereas in Theorem~\ref{bigthm2} we required superhypergraphs of $L(H)$.

\begin{proof}

Assume w.l.o.g.\ that $\epsilon < 1$.

We aim to assign $M$ a parent matroid $T \in
\mathbb{S}_{n,r,m,H}$, so that every choice of $M$ has precisely one such parent. We
show then that for any parent, at most proportion $\epsilon$ of its children are
$H$-free. The same will then be true for the union of offspring of all parents,
which is of course equal to the entirety of $\mathbb{S}_{n,r,m,H}$.

Assume w.l.o.g.\ that our groundset is $[n]$. Take a lexicographical ordering across
all the $n_H$-subsets of $[n]$. Now let  $\mathcal{U}(n,m,H)$ be the set of all
possible collections of $m$ or more pairwise disjoint $n_H$-subsets of $[n]$. Now
let $*(n,m,H)$ be the lexicographical ordering on  $\mathcal{U}(n,m,H)$ induced by
our lexicographical ordering on the $n_H$-subsets . Say that $V >^* W$ if $V$ is
above $W$ in this ordering.

Note that in fact any total orderings would do here, but as in our earlier proof it
is most intuitive to imagine a lexicographical ordering.

Now for any $U = \{U_0, \ldots, U_{m-1}\}\in \mathcal{U}(n,m,H)$, and $0 \le k \le
m-1$, let us define $\mathcal{T}_{U}$ to be the collection of
all stable sets $T \in J(n,r)$ such that in $M(T)$ we have: \\

\begin{itemize}
\item[(i)]$U$ forming a maximal collection of element disjoint $H$-good moats under
$>^*$
\item[(ii)]For all $0 \le k \le m$, $ U_k$ containing a copy of $L(H)$
\end{itemize}

Note each $U_k$  must contain only a copy of $L(H)$, since by (i) it is an $H$-good
moat.

One last consideration: let's also establish a total ordering over all possible
copies of $L(H)$ (for example taking the lexicographical ordering induced by the
lexicographical ordering on $r$-subsets of $[n]$.

Now, we have at least one collections of $m$ element disjoint $H$-good moats in
$M$, and let's say that $V(M) = \{V_0, \ldots, V_{m-1}\}$ is the maximal such
collection under our ordering
relation $>^*$. The parent of $M$, which we may denote $T(M)$, shall be $T \in
\mathcal{T}_{V}: \mathcal{C}(T) =
\mathcal{C}(M) \cup A_0 \cup \ldots \cup A_{m-1}$, with $A_k$ the maximal copy of
$L(H)$ in $V_k$
such that $\mathcal{C}(M|V_k) \subseteq A_k$.

Let $X_k$ be the event that the restriction of $M$ to $V_k$
forms a $H$-minor, and $Y_i$ be the converse. Further, let $Z_k = Y_0 \cup \ldots
\cup Y_k$. Note that in the event $M$ is $H$-free,  $Z_{m-1}$ is implied. So 

We aim to show that

\begin{equation}
\label{Je5s}
\mathrm{Pr}(Z_{m-1}) \le \left(\frac{2^{|L(H)|}-1}{2^{|L(H)|}}\right)^m
\end{equation}

Having shown this we will be done:  for any $\epsilon > 0$ we can achieve
$\mathrm{Pr}(M \not\succeq H) <
\epsilon$. We simply need $m > -\log \epsilon/(\log(2^{|L(H)|}-1) - \log
(2^{|L(H)|}))$
and this provides our value for $m(\epsilon)$.

To prove \ref{Je5s} it suffices to show that, for all $0 \le k \le m-1$, 

$$\mathrm{Pr}(X_k|Z_{k-1}) \ge
\frac{1}{2^{|L(H)|}}$$

For any $U$ and any $T \in \mathcal{T}_{U}$, the children of $T$ are contained in
the collection of sparse paving matroids which can be formed by taking $T$ and
removing lines contained in any of $\{U_0, \ldots, U_{m-1}\}$. Depending on the
choice of lines to remove, we may or may not be left with a child of $T$. Now
consider making any choice of lines to remove from $\{U_0, \ldots, U_{k-1}, U_{k+1},
\ldots, U_m \}$. Suppose that having done so we are still able to find some choice
of lines to remove from $U_k$ such that we are left with a child of $T$. Then
certainly removing no lines from $U_k$ also leaves us with a child of $T$, and
moreover one which contains an $H$-minor! This accounts for one of a maximum of
$2^{|L(H)|}$ choices of lines that can be removed from $U_k$ to leave a child of
$T$. Since this holds regardless of our choice of lines removed from the other
$U_i$, we can immediately say that

$$\mathrm{Pr}(X_k|T(M)=T) \ge \frac{1}{2^{|L(H)|}}$$

But we may apply exactly the same argument even if we add the restriction that our
choice of lines to be removed contain at least one line from each of $U_0, \ldots,
U_{k-1}$. And so 

$$\mathrm{Pr}(X_k|Z_{k-1},T(M)=T) \ge \frac{1}{2^{|L(H)|}}$$

And now, writing $\mathcal{U}$ for $\mathcal{U}(n,m,H)$, we have

$$\mathrm{Pr}(X_k|Z_{k-1}) = \sum\limits_{U \in \mathcal{U}}\sum\limits_{T
\in\mathcal{T}_U}\mathrm{Pr}(X_k|Z_{k-1},T(M)=T)\mathrm{Pr}(T_(M)=T|Z_{k-1})  $$

$$ \ge \sum\limits_{U \in \mathcal{U}}\sum\limits_{T
\in\mathcal{T}_U}\frac{1}{2^{|L(H)|}}\mathrm{Pr}(T(M)=T|Z_{k-1}) $$
$$ \ge \frac{1}{2^{|L(H)|}}\sum\limits_{U \in \mathcal{U}}\sum\limits_{T \in
\mathcal{T}_U}\mathrm{Pr}(T(M)=T|Z_{k-1}) =\frac{1}{2^{|L(H)|}}$$

\end{proof}

Note that though we have chosen only to prove the result for sparse paving matroids,
we believe the same result could be proven for general matroids in essentially the
same way.

Now we want to deal with forts. Recall our earlier definition: let $M$ be a sparse
paving matroid of rank $r$. A collection of elements $X$ forms a
\textit{fort} of $M$ if for every $A \in X^{(r-1)}, \exists C \in \mathcal{C}(M): C
= A \cup e, e \not\in X$.

We want to show the following. \\

\begin{lemma}
\label{fortlemma}
Suppose $X$ is a fort of size $n$ in a sparse paving matroid $M$. For any $A
\subseteq X$, Let $\mathcal{C}_A$ be the non-bases of $M$ which contain $r-1$
elements of $A$. Then for any $m \ge r$, there exists $n(m,r)$ such that $n > n(m,r)$
implies that there exists $X' \subset X, |X| \ge m$, such that no two elements of
$\mathcal(C)_{X'}$  intersect outside of $X'$.
\end{lemma}

That is to say that given an arbitrarily large fort in a sparse paving matroid $M$
of rank $r$, we can find an arbitrarily large subset $X'$ of that fort on which each
$(r-1)$-sets of $X'$ can be matched with a unique element of $M \ X'$ to form a
non-basis of $M$.

First we need to prove a result that is in some sense anti-Ramsey, that is to say we
want to demonstrate the existence of \textit{polychromatic} cliques in a hypergraph,
given sufficient size and one simple condition.

\begin{defn}[Polychromatic]
Let $G$ be an $r$-uniform hypergraph and $c: E(G) \rightarrow \mathbb{N}$ be a
colouring of the hyperedges of $G$. A subhypergraph $H$ is \textit{polychromatic} if
$A, B \in E(H) \implies c(A) \neq c(B)$.
\end{defn} 

\begin{lemma}
\label{antiram}
Let X be a set of $n$ elements. Say we have a function $c: X^{(r)} \rightarrow
\mathbb{N}$ such that $A, B \in X^{(r)}, |A \cap B| = r -1 \implies c(A) \neq c(B)$.
Then for any $m \ge r$, there exists $n(m,r)$ such that $n > n(m,r) \implies \exists
X' \subset X, |X'| = m, A, B \in X'^{(r)} \implies c(A) \neq c(B)$.
\end{lemma}

We note that our condition is equivalent to specifying that $c$ is a vertex-colouring
of the Johnson graph $J(n,r)$. However it makes more sense for us to view $c$ as a
hyperedge-colouring of a $r$-uniform hypergraph.

\begin{proof}
We prove this by (transfinite) induction on $m$ and $r$, ordering on $r$ first. Thus
we say $(k,s) < (m,r)$ if either$s < r$, or $s = r$ and  $k <m$. We take as our
inductive hypothesis that $n(k,s)$ exists for all $(k,s) < (m,r)$. 
We can see easily enough that $n(k,1) = k$ and $n(1,r) = r$, which covers all base
and limit cases in our ordering. Next we use our inductive hypothesis and proceed to
demonstrate the existence of $n(m,r)$ in the case $m, r \neq 0$.

First we choose some element $x_1 \in X$. We now consider hyperedges containing
$x_1$. Since each of these is the union of $x_1$ with $r-1$ other elements, we can
use the existence of $n(k,r-1)$, for all $k$, to guarantee the existence of an
arbitrarily large
set $X_1$ such that $A,B \in X_1^{(r)}, x_1 \in A, B \implies c(A) \neq c(B)$. We
now choose $x_2 \in X_1$ and repeat to produce $X_2$ such that  $A,B \in X_2^{(r)},
\{x_1,x_2\} \subseteq A, B \implies c(A) \neq c(B)$. We can continue in such a way
any eventually build a set $ X_{(a)} = \{x_1, x_2, \ldots, x_{n(m-1,r)}\}$ which now
has
the property that $A, B \in X_{(a)}^{(r)}, A \cap B \neq \emptyset \implies c(A) \neq
c(B)$. We now invoke the existence of $n(m-1,r)$ to fsay there must be some $X_{(b)}
\subseteq X_{(a)}$, $|X_{(b)}| = m-1$ with the
additional property that $A, B \in X_{(b)} \implies c(A) \neq c(B)$. 

We now seek to form $X'$ by extending $X_{(b)}^{(r)}$ within $X_{(a)}$. Let
$\mathcal{A}$
and $\mathcal{B}$ be defined as
$$\mathcal{A} = \{A: A \in X_{(a)}^{(r)}, |A \cap X_{(b)} = r-1\}$$
$$\mathcal{B} = X_{(b)}^{(r)}$$

Now no two of the sets in $\mathcal{A}$ can have the same colour if they share an
element, and so at most $\frac{m-1}{r-1}$ of these have any given colour (else some
element of $X_{(b)}$ would appear in two of them, a contradicition). At most ${m-1
\choose
r}\frac{m-1}{r-1}$ members of $\mathcal{A}$ share a colour with a member of
$\mathcal{B}$. But now if we force $|X_{(a)}| \ge m + {m-1 \choose
r}\frac{m-1}{r-1}$, we have that there must be some element $e \in X_{(a)}
\backslash X_{(b)}$ such
that no member of $\mathcal{A}$ containing $e$ shares a colour with any member of
$\mathcal{B}$. Now setting $X' = X_{(b)} \cup e$ meets our demand that no two
members of $X'^{(r)}$ share a colour.
\end{proof}

Note that this proof does not give a practical bound on $n(m,r)$ - indeed let $M$ be
the required size of $X_{(a)}$, which is equal to $n(m-1,r)$ or $ (m+1) + {m-1
\choose r}\frac{m-1}{r-1}$, whichever is greater. Then we require $n > n^M(M,r)$,
where $n^a(b,c)$ is notation for $n(n(n(n(\ldots(b,c),c)\ldots),c)$ with $a$
iterations.

\begin{proof}[Proof of Lemma~\ref{fortlemma}]
Let $X$ be a fort of size $n$ in a sparse paving matroid of rank $r$. Apply
Lemma~\ref{antiram} on $X^{(r-1)}$ colouring each member by the unique element with
which
it is paired in a non-basis of $M$. So long as $n > n(m,r-1)$ we can find a
suitable $X'$.
\end{proof}

We make one important observation based on Lemma~\ref{fortlemma}, which is that if
setting $m = n_H-|L(H)|$ and finding a $X'$ of cardinality at least $n_H-|L(H)|$,
then necessarily $\mathcal{C}(X')$ contains a copy of $L(H)$.

\begin{proof}[Proof of Theorem~\ref{aaaloose}]
Let $n(m,r)$ be defined as in Lemma~\ref{antiram}.

Let $M$ be a sparse paving matroid of rank $r$ on $n$ elements and allow $n$ to grow
arbitrarily large. We consider the $r-1$-uniform hypergraph on $[n]$ formed by
including a hyperedge if and only if it is contained in a circuit hyperplane of $M$.
Note that a stable set in this hypergraph represents a moat. Moreover if we have a
sufficiently large complete subgraph then, by applying Ramsey to the $r$-uniform
hypergraph on the same vertices in a similar manner to Theorem~\ref{uniframseysp}
then we can find within that complete subgraph a fort of any given size. So, by
using Ramsey at the $r-1$ level we can find either a moat of size $R_r(r+1,n_H)$ or
a fort
of size $n(n_H-|L(H)|,r)$, so long as $n > n_0 $ (where $n_0$ may be no more than
$R_{r-1}(R_r(r+1,n_H),R_r(r+1,n(n_H-|L(H)|,r))$).

Now we imagine that $n$ is truly enormous, say $n = Kn_0$ .Now $[n]$ can be divided
into intervals 

$$[1,n_0],[n_0+1,2n_0], \ldots, [K-1)n_0+1,Kn_0]$$ 

We note that by Ramsey each of these intervals must contain either a fort of
cardinality $n(n_H-|L(H)|,r)$ or a moat of cardinality $R_r(r+1,n_H)$. In the
former case we
have a copy of $L(H)$, by Lemma \ref{fortlemma}. In the latter case we must have
(again by the argument of Theorem \ref{uniframseysp}) some $U_{r,n_H}$ minor within
our moat - and this then forms an \textit{empty} moat (and hence also an $H$-good
moat) on those $n_H$ elements! So by making $K$ large enough we can force that $M$
contains either arbitrarily many
copies of $L(H)$, or contains arbitrarily many $H$-good $n_H$-moats. Then by
applying either Lemma
\ref{bigthmfixrank} or Lemma~\ref{moatlemma} respectively we can see that
asymptotically almost certainly $M$ will contain $H$ as a minor.

\end{proof}

\begin{corollary}
\label{aaalooseallr}
Let $H$ be a sparse paving matroid with no tied non-bases, and $r > r_H$. Then
asymptotically almost
all matroids of rank $r$ contain $H$ as a minor.
\end{corollary}

\begin{proof}
It suffices to observe that $H$ is contained as a minor of a matroid of rank $r$
which also has no tied non-bases. To see this, simply form $H'$ by adding the set of
elements $A = \{a_1, a_2, \ldots, a_{r-r_H}\}$ to the groundset of $H'$, and define
$\mathcal{C}(H') = \{C \cup A: C \in \mathcal{C}(H)\}$
\end{proof}

Of course the above is a large amount of work for a relatively modest result. One
hopes however that some of the methods involved, in particular those involving the
moats and forts, might be of use in proving other results.

\section{Acknowledgement}

I'd like to thank my supervisor Dillon Mayhew for the frequent discussion, advice
and suggestions he contributed during the course of this work.

\vfill

\end{document}